\documentclass[a4paper]{article}
\usepackage{amssymb, amsmath, amsthm, amsmath}
\usepackage{mathrsfs}
\usepackage{latexsym}
\usepackage{enumerate}
\usepackage{datetime}

\usepackage{hyperref}

\newtheorem{theorem}{Theorem}
\newtheorem{lemma}{Lemma}
\newtheorem{corollary}{Corollary}
\newtheorem{proposition}{Proposition}

\newtheorem{definition}{Definition}
\newtheorem{question}{Question}
\newtheorem{remark}{Remark}

\newcommand{\A}{\mathcal{A}}
\newcommand{\B}{\mathcal{B}}

\newcommand{\D}{\mathcal{D}}
\newcommand{\K}{\mathfrak{K}}
\newcommand{\E}{\mathcal{E}}
\newcommand{\Pw}{\mathcal{P}(\omega)}

\newcommand{\InfEx}{\mathbf{InfEx}}
\newcommand{\TxtEx}{\mathbf{TxtEx}}

\newcommand{\Inf}{\mathbb{I}}
\newcommand{\Dp}{\mathcal{D}_+}

\newcommand{\embedCantor}{\leq_{\textnormal{Cantor}}}
\newcommand{\embedScott}{\leq_{\textnormal{Scott}}}

\usepackage{color}
\usepackage{thmtools}
\usepackage{scalerel}
\usepackage{stmaryrd}
\newcommand\vvee{\hstretch{.8}{\vee\mkern-8mu\vee}}       
\newcommand\wwedge{\hstretch{.8}{\wedge\mkern-8mu\wedge}} 

\newcommand{\df}{\stackrel{\text{def}}{=}}

\newcommand{\pair}[1]{\langle{#1}\rangle}
\newcommand{\ov}[1]{\overline{#1}}

\usepackage[backend=biber,style=numeric,url=false,maxbibnames=10]{biblatex}
\usepackage{authblk}

\author[1]{Nikolay Bazhenov}
\author[2]{Ekaterina Fokina}
\author[2]{Dino Rossegger}
\author[3]{Alexandra Soskova}
\author[3]{Stefan Vatev}

\affil[1]{Sobolev Institute of Mathematics, Novosibirsk}
\affil[ ]{bazhenov@math.nsc.ru}
\affil[2]{Institute of Discrete Mathematics and Geometry, Technische Universität Wien}
\affil[ ]{\{dino.rossegger,ekaterina.fokina\}@tuwien.ac.at}
\affil[3]{Faculty of Mathematics and Informatics, Sofia University}
\affil[ ]{\{asoskova,stefanv\}@fmi.uni-sofia.bg}

\title{Learning Families of Algebraic Structures from Text
}
\date{}

\begin{document}

\maketitle

\begin{abstract}
  We adapt the classical notion of learning from text to computable structure theory.
  Our main result is a model-theoretic characterization of the learnability 
  from text for classes of structures.
  We show that a family of structures is learnable from text if and only if the structures can be distinguished in terms of their theories restricted to positive infinitary $\Sigma_2$ sentences.
\end{abstract}

\section{Introduction}

The classical algorithmic learning theory goes back to the works of Putnam~\cite{putnam-65} and Gold~\cite{gold-67}. A learner $M$ receives step by step more and more data (finite amount at each step) on an object $X$ to be learned, and $M$ outputs a sequence of hypotheses that converges to a finitary description of $X$. The classical studies (up to the beginning of 2000s) mainly focused on learning for formal languages and for recursive functions, see the monograph~\cite{jain-learning-book}.

Within the framework of computable structure theory, the work of Stephan and Ventsov~\cite{Stephan-Ventsov} initiated investigations of learnability for classes of substructures of a given computable structure $\mathcal{S}$. This approach was further developed, e.g., in the papers~\cite{Harizanov-Stephan,Gao-Stephan}. 

Fokina, K{\"o}tzing, and San Mauro~\cite{pmlr-v98-fokina19a} considered various
classes $\mathfrak{K}$ of computable equivalence relations. For these
$\mathfrak{K}$, they introduced the notions of \emph{learnability from
informant} (or $\InfEx$-learnability) and \emph{learnability from text}
($\TxtEx$-learnability). The work~\cite{bazhenov_learning_2020} extended the
notion of $\InfEx$-learnability to arbitrary countable families of computable
structures 
and obtained the following general model-theoretic characterization of $\InfEx$-learnability. Let $\mathfrak{K} = \{ \mathcal{A}_i : i \in\omega \}$ be a family of computable structures such that $\mathcal{A}_i \not\cong \mathcal{A}_j$ for $i\neq j$. Then $\mathfrak{K}$ is learnable from informant if and only if there exists a family of infinitary $\Sigma_2$ sentences $\{ \psi_i : i\in\omega\}$ such that
\begin{itemize} 
	\item[($\dagger$)] for each $i$, $\mathcal{A}_i$ is the only member of $\mathfrak{K}$ satisfying $\psi_i$. 
\end{itemize}
In turn, the results of~\cite{bazhenov_learning_2020} led to discovering some unexpected connections between $\InfEx$-learnability and results from descriptive set theory, see~\cite{BCSM-21}.

\medskip

Until now, to our best knowledge, there was no notion of $\TxtEx$-learnability applicable for general classes of countable structures. This paper aims to close this gap. In Section~\ref{sect:formal-framework}, we introduce our new formal framework for classes of structures $\mathfrak{K}$: this approach allows us to simultaneously give both the known definition of $\InfEx$-learnability and the new definition of $\TxtEx$-learnability.

The main result of the paper (Theorem~\ref{th:txtex-learning-equivalence}) shows that $\TxtEx$-learnability admits a model-theoretic characterization similar to the characterization of $\InfEx$-learnability discussed above: a family $\{ \mathcal{A}_i : i \in\omega \}$ is $\TxtEx$-learnable if and only if there exists a family of $\Sigma^p_2$ sentences $\{\psi_i : i\in\omega\}$ satisfying ($\dagger$). Here $\Sigma^p_n$ formulas are \emph{positive infinitary $\Sigma_n$ formulas} introduced in our previous work~\cite{bazhenov_lopez-escobar_2023}, see the formal details in Section~\ref{sect:Scott-Continuous}.



\section{The Formal Framework} \label{sect:formal-framework}

Let us consider structures $\A$ with domains a subset of $\omega$.
We consider computable signatures with $=$ and $\neq$.
We shall denote by $\D(\A)$ the basic diagram of $\A$, i.e., $\D(\A)$ contains exactly the positive and negative atomic sentences true in $\A$,
and by $\Dp(\A)$ the positive atomic diagram of $\A$, i.e., $\Dp(\A)$ contains only the positive atomic sentences true in $\A$.

For a signature $L$, by $Mod(L)$ we denote the set of all $L$-structures $\mathcal{A}$ with $\mathrm{dom}(\mathcal{A}) \subseteq \omega$. If not specified otherwise, we assume that every considered class $\mathfrak{K} \subseteq Mod(L)$ is closed under isomorphisms.

\medskip

First we need to introduce the components of our learning framework. 
Let $\mathfrak{K} \subseteq Mod(L)$ be a family which contains precisely $\kappa$ isomorphism types, where $\kappa \leq \omega$, those are the types of $L$-structures $\mathcal{A}_i$, $i\in\kappa$.
\begin{itemize}
\item
  The \emph{learning domain} (LD) is the collection of all copies $\mathcal{S}$ of the structures from $\K$ such that $\mathrm{dom}(\mathcal{S}) \subseteq \omega$, i.e., 
  \[\mathrm{LD}(\mathfrak{K}) = \bigcup_{i\in \kappa} \{ \mathcal{S} \in Mod(L) : \mathcal{S} \cong \mathcal{A}_i\}.\]

\item
  The \emph{hypothesis space} (HS) contains the indices $i$ for $\mathcal{A}_i \in \mathfrak{K}$ (an index is viewed as a conjecture about the isomorphism type of an input structure $\mathcal{S}$) and a question mark symbol:
  \[\mathrm{HS}(\mathfrak{K}) = \kappa \cup \{ ?\}.\]
\item
  A \emph{learner} $M$ sees, stage by stage, some atomic facts about a given structure from $LD(\mathfrak{K})$.
  The learner $M$ is required to output conjectures from $HS(\mathfrak{K})$. This is formalized as follows. 
  \smallskip

  Let $Atm$ denote the set of (the G{\"o}del numbers) of all positive and negative atomic sentences in the signature $L\cup \omega$ (in other words, positive and negative atomic facts about possible $L$-structures on the domain $\omega$). The restriction of $Atm$ to only positive atomic sentences is denoted by $Atm_{+}$.	
  A \emph{learner} $M$ is a function from the set $(Atm)^{<\omega}$ (i.e., the set of all finite tuples of atomic facts) into $\mathrm{HS}(\mathfrak{K})$.
  
  \smallskip
\item
  For an $L$-structure $\mathcal{S}$, an \emph{informant} $\mathbb{I}$ for $\mathcal{S}$ is an arbitrary sequence $(\psi_0,\psi_1,\psi_2,\dots)$ containing elements from $Atm$ and satisfying 
  \[\D(\mathcal{S}) = \{ \psi_i : i\in\omega\}.\]
\item
  For an $L$-structure $\mathcal{S}$, a \emph{text} $\mathbb{T}$ for $\mathcal{S}$ is an arbitrary sequence $(\psi_0,\psi_1,\psi_2,\dots)$ containing elements from $Atm_{+}$ and satisfying 
  \[\Dp(\mathcal{S}) = \{ \psi_i : i\in\omega\}.\]
\item
  For $k\in\omega$, by $\mathbb{I}\upharpoonright k$ (respectively, $\mathbb{T}\upharpoonright k$) we denote the corresponding sequence $(\psi_i)_{i<k}$.
\end{itemize}

\begin{definition}[\cite{bazhenov_turing_2021}]
  We say that the family $\K$ is \emph{$\InfEx$-learnable} if there exists a learner $M$ such that for any structure $\mathcal{S}\in \mathrm{LD}(\K)$ and any informant $\Inf_{\mathcal S}$ for $\mathcal{S}$, the learner eventually stabilizes to a correct conjecture about the isomorphism type of $\mathcal{S}$. More formally, there exists a limit
  \[\lim_{n\to \omega} M(\Inf_{\mathcal{S}}\upharpoonright {n}) = i\]
  belonging to $\omega$, and $\mathcal{A}_i$ is isomorphic to $\mathcal{S}$.
\end{definition}

Recall that a structure $\mathcal{A} = (A; \sim)$ is an \emph{equivalence structure} if $\sim$ is an equivalence relation on $A$. 
The paper~\cite{pmlr-v98-fokina19a} introduced the definition of $\TxtEx$-learnability for equivalence structures. Here we generalize this definition to arbitrary structures.

\begin{definition}
  We say that the family $\mathfrak{K}$ is \emph{$\TxtEx$-learnable} if there exists a learner $M$ such that for any structure $\mathcal{S}\in \mathrm{LD}(\mathfrak{K})$ and any text $\mathbb{T}_{\mathcal S}$ for $\mathcal{S}$, the learner eventually stabilizes to a correct conjecture about the isomorphism type of $\mathcal{S}$. More formally, there exists a limit
  \[\lim_{n\to \omega} M(\mathbb{T}_{\mathcal{S}}\upharpoonright {n}) = i\]
  belonging to $\omega$, and $\mathcal{A}_i$ is isomorphic to $\mathcal{S}$.
\end{definition}

In this paper, we give many examples of classes of equivalence structures. 
We use the notation $[\alpha_1:\beta_1,\dots,\alpha_n:\beta_n]$, where $\alpha_i,\beta_i \leq \omega$,
to denote the equivalence structure with precisely $\beta_i$-many equivalence classes of size $\alpha_i$, for all $i=1,\dots,n$  (and with no equivalence classes of other sizes).

\begin{remark}\label{remark:the-key-classes}
The classes of equivalence structures
$\E = \{[\omega:1, n:1]\ \mid\ n\geq 1\}$ and $\tilde\E = \{[\omega:\omega, n:\omega] \ \mid\ n\geq 1\}$
play an important role in this paper.
\end{remark}

\begin{remark}
	 It is easy to observe that every $\TxtEx$-learnable class is also
   $\InfEx$-learnable (indeed, notice that an informant $\mathbb{I}$ for a
   structure $\A$ can be effectively transformed into a text $\mathbb{T}_I$ for
   this $\A$). Theorem 1.4 in~\cite{pmlr-v98-fokina19a} proves that the class $\mathfrak{K} = \{[\omega:1], [\omega:2]\}$ is $\InfEx$-learnable, but not $\TxtEx$-learnable.
\end{remark}



\section{Cantor-Continuous Embeddings}


  

The Cantor space, denoted by $2^\omega$, can be represented as the collection of reals, equipped with the product topology of the discrete topology on the set $\{0,1\}$.
A basis for $2^\omega$ is formed by the collection of $[\sigma] = \{f \in 2^\omega : \sigma \subset f\}$, for all finite binary strings $\sigma$.
Here we will need the following characterization of the Cantor-continuous functions.

\begin{proposition}[Folklore]\label{prop:cantor:turing:relative}
  A function $\Psi: 2^\omega \to 2^\omega$ is Cantor-continuous if and only if
  there exists a Turing operator $\Phi_e$ and a set $A \in 2^\omega$ such that $\Psi(X) = \Phi_e(A\oplus X)$ for all $X \in 2^\omega$.
\end{proposition}

\begin{definition}\label{def:cantor-embedding}
  For $i\in\{0, 1\}$, let $\mathfrak{K}_i$ be a class of $L_i$-structures. A mapping $\Psi$ is a \emph{Cantor-continuous embedding} of $\mathfrak{K}_0$ into $\mathfrak{K}_1$,
  denoted by $\Psi\colon \mathfrak{K}_0 \embedCantor \mathfrak{K}_1$, if $\Psi$ is Cantor-continuous and satisfies the following:
  \begin{enumerate}
  \item
    For any $\A \in \mathfrak{K}_0$, $\Psi(\D(\A))$ is the characteristic function of the atomic diagram of a structure from $\mathfrak{K}_1$.
    This structure is denoted by $\Psi(\A)$.
  \item
    For any $\A, \B \in\mathfrak{K}_0$, we have $\A \cong \B$ if and only if $\Psi(\A) \cong \Psi(\B)$.
  \end{enumerate}
\end{definition}

When the embedding $\mathfrak{K}_0 \embedCantor \mathfrak{K}_1$ is given by a Turing operator $\Phi_e$, then we say that
$\mathfrak{K}_0$ is \emph{Turing computable embeddable} into $\mathfrak{K}_1$, and we denote this by $\Phi_e \colon \mathfrak{K}_0 \leq_{tc} \mathfrak{K}_1$.
The study of this notion was initiated in \cite{calvert_knight_2004,knight_turing_2007}. 
One of the main tools in proving results about the Turing computable embeddability is the following Pullback Theorem.
Here the $\Sigma^c_\alpha$ formulas are the usual computable infinitary $\Sigma_\alpha$ formulas as defined in \cite{ash_knight_book}. A $\Sigma^{\text{inf}}_{\alpha}$ formula is an infinitary $\Sigma_\alpha$ formula.

\begin{theorem}[Pullback Theorem \cite{knight_turing_2007}]\label{th:turing:pullback}
  Let $\Phi_e \colon \mathfrak{K} \leq_{tc} \mathfrak{K}'$.
  Then for any computable infinitary sentence $\varphi'$ in the signature of $\mathfrak{K}'$,
  we can \emph{effectively} find a computable infinitary sentence $\varphi$ in the signature of $\mathfrak{K}$ such that for all $\A \in \mathfrak{K}$,
  \[\A \models \varphi \text{ if and only if } \Phi_e(\A) \models \varphi'.\]
  Moreover, for a nonzero $\alpha < \omega^{CK}_1$, if $\varphi'$ is $\Sigma^{c}_{\alpha}$ (or $\Pi^{c}_{\alpha}$), then so is $\varphi$.
\end{theorem}

As noted in \cite{bazhenov_learning_2020}, Theorem~\ref{th:turing:pullback}, can be relativized to an arbitrary oracle $X$.
By Proposition~\ref{prop:cantor:turing:relative}, we directly obtain the following non-effective version of Theorem~\ref{th:turing:pullback}.

\begin{corollary}[Non-effective Pullback Theorem]\label{cor:cantor:pullback}
  Let $\Psi \colon \K \embedCantor \K'$.
  Then for any infinitary sentence $\varphi'$ in the signature of $\K'$,
  there exists an infinitary sentence $\varphi$ in the signature of $\K$ such that for all $\A \in \K$,
  \[\A \models \varphi \text{ if and only if } \Psi(\A) \models \varphi'.\]
  Moreover, for a nonzero $\alpha < \omega_1$, if $\varphi'$ is $\Sigma^{\textnormal{inf}}_{\alpha}$ (or $\Pi^{\textnormal{inf}}_{\alpha}$), then so is $\varphi$.
\end{corollary}

\section{Scott-Continuous Embeddings} \label{sect:Scott-Continuous}

The Scott topology, denoted by $\Pw$, can be characterized as the product topology of the Sierpi{\'n}ski space on $\{0,1\}$.
The Sierpi{\'n}ski space on $\{0,1\}$ is the topological space with open sets $\{\emptyset, \{1\}, \{0,1\}\}$.
A basis for $\Pw$ is formed by the collection $[D] = \{A \subseteq \omega : D \subseteq A\}$, for all finite sets $D$.

\begin{definition}[Case \cite{case_enumeration_1971}]
  A set $A \in \Pw$ defines a \emph{generalized enumeration operator} $\Gamma_A : \Pw \to \Pw$ if and only if for each set $B \in \Pw$,
  \[\Gamma_A(B) = \{x : \exists v (\pair{x,v} \in A\ \&\ D_v \subseteq B)\}.\]
\end{definition}
When $A = W_e$ for some c.e.\ set $W_e$, we write $\Gamma_e$ instead of $\Gamma_{W_e}$, which is the usual enumeration operator as defined in \cite{case_enumeration_1971} and \cite{cooper_enumeration_1990}, for example.

\begin{proposition}[Folklore]\label{prop:scott:enumeration}
  A mapping $\Gamma: \Pw \to \Pw$ is Scott-continuous if and only if $\Gamma$ is a generalized enumeration operator.
\end{proposition}

As a direct corollary of Proposition~\ref{prop:scott:enumeration}, the following characterization will be useful.
\begin{corollary}\label{cor:scott-continuous:characterization}
  A mapping $\Psi:\Pw \to \Pw$ is Scott-continuous if and only if $\Psi$ is
  \begin{enumerate}[(a)]
  \item
    monotone, i.e., $A \subseteq B$ implies $\Psi(A) \subseteq \Psi(B)$, and
  \item
    compact, i.e., $x \in \Psi(A)$ if and only if $x \in \Psi(D)$ for some finite $D \subseteq A$.
  \end{enumerate}
\end{corollary}

We define Scott-continuous embedding for classes of structures as an analogue of the Cantor-continuous embedding from Definition~\ref{def:cantor-embedding}.
Here we take into consideration only the \emph{positive} atomic diagram $\Dp(\A)$ of a structure $\A$,
and not the basic (positive and negative) atomic diagram $\D(\A)$ as in \cite{knight_turing_2007}.

\begin{definition}\label{def:scott-embedding}
  A mapping $\Gamma$ is a \emph{Scott-continuous embedding} of $\mathfrak{K}_0$ into $\mathfrak{K}_1$,
  denoted by $\Gamma\colon \mathfrak{K}_0 \embedScott \mathfrak{K}_1$, if $\Gamma$ is Scott-continuous and satisfies the following:
  \begin{enumerate}
  \item
    For any $\A \in \mathfrak{K}_0$, $\Gamma(\Dp(\A))$ is the positive atomic diagram of a structure from $\mathfrak{K}_1$.
    This structure is denoted by $\Gamma(\A)$.
  \item
    For any $\A, \B \in\mathfrak{K}_0$, we have $\A \cong \B$ if and only if $\Gamma(\A) \cong \Gamma(\B)$.
  \end{enumerate}
\end{definition}

If we consider enumeration operators $\Gamma_e$, we obtain an effective version of Definition~\ref{def:scott-embedding}.
We say that $\Gamma_e$ is a \emph{positive computable embedding} of $\mathfrak{K}_0$ into $\mathfrak{K}_1$,
and we denote it by $\Gamma_e \colon \mathfrak{K}_0 \leq_{pc} \mathfrak{K}_1$. Positive computable embeddings were first studied in \cite{bazhenov_lopez-escobar_2023}.
To obtain an analogue of Theorem~\ref{th:turing:pullback} for positive computable embeddings, we need to define a hierarchy of \emph{positive} infinitary formulas.

\begin{definition}[\cite{bazhenov_lopez-escobar_2023}]\label{def:posinf}
  Fix a countable signature $L$. For every $\alpha<\omega_1$ define the sets of $\Sigma^p_\alpha$ and
  $\Pi^p_\alpha$ $L$-formulas inductively as follows.
\begin{itemize}
\item
  Let $\alpha = 0$. Then:
  \begin{itemize}
  \item
    the $\Sigma^p_0$ formulas are the finite conjunctions of atomic $L$-formulas.
  \item
    the $\Pi^p_0$ formulas are the finite disjunctions of \emph{negations} of
    atomic $L$-formulas.
  \end{itemize}
\item
  Let $\alpha = 1$. Then:
  \begin{itemize}
  \item
    $\varphi(\bar{u})$ is a $\Sigma^p_1$ formula if it has the form
    \[\varphi(\bar{u}) = \vvee_{i\in I} \exists \bar{x}_i \psi_i(\bar{u},\bar{x}_i), \]
    where for each $i \in I$, $\psi_i(\bar{u},\bar{x}_i)$ is a $\Sigma^p_0$ formula, $I$ is countable.
  \item
    $\varphi(\bar{u})$ is a $\Pi^p_1$ formula if it has the form
    \[\varphi(\bar{u}) = \wwedge_{i\in I} \forall \ov{x}_i \psi_i(\bar{u},\bar{x}_i), \]
    where for each $i \in I$, $\psi_i(\bar{u},\bar{x}_i)$ is a $\Pi^p_0$ formula, $I$ is countable.    
  \end{itemize}
\item
  Let $\alpha \geq 2$. Then:
  \begin{itemize}
  \item
    $\varphi(\bar{u})$ is $\Sigma^p_{\alpha}$ formula if it has the form
    \[\varphi(\bar{u}) = \vvee_{i\in I} \exists \bar{x}_i(\xi_i(\bar{u},\bar{x}_i) \land \psi_i(\bar{u},\bar{x}_i)), \]
    where for each $i \in I$, $\xi_i(\bar{u},\bar{x}_i)$ is a $\Sigma^p_{\beta_i}$ formula and $\psi_i(\bar{u},\bar{x}_i)$ is a $\Pi^p_{\beta_i}$ formula, for some $\beta_i < \alpha$ and $I$ countable.
  \item
    $\varphi(\bar{u})$ is $\Pi^p_{\alpha}$ formula if it has the form
    \[\varphi(\bar{u}) = \wwedge_{i\in I} \forall \bar{x}_i(\xi_i(\bar{u},\bar{x}_i) \lor \psi_i(\bar{u},\bar{x}_i)), \]
    where for each $i \in I$, $\xi_i(\bar{u},\bar{x}_i)$ is a $\Sigma^p_{\beta_i}$ formula and $\psi_i(\bar{u},\bar{x}_i)$ is a $\Pi^p_{\beta_i}$ formula, for some $\beta_i < \alpha$ and $I$ countable.
  \end{itemize}
\end{itemize}
\end{definition}

A similar hierarchy of positive infinitary formulas can be also found in \cite{soskov_degree_2004}, where it is used in connection with the $\alpha$-th enumeration jump.
As usual, when we restrict to c.e.\ index sets $I$ in the above definition, we obtain the hierarchy of \emph{positive computable infinitary formulas}.
For this hierarchy, we will use the notations $\Sigma^{pc}_\alpha$ and $\Pi^{pc}_\alpha$.

\begin{theorem}[Pullback Theorem \cite{bazhenov_lopez-escobar_2023}]\label{th:scott:pullback}
  Let $\Gamma_e \colon \mathfrak{K} \leq_{pc} \mathfrak{K}'$.
  Then for any positive computable infinitary sentence $\varphi'$ in the signature of $\mathfrak{K}'$,
  we can \emph{effectively} find a positive computable infinitary sentence $\varphi$ in the signature of $\mathfrak{K}$ such that for all $\A \in \mathfrak{K}$,
  \[\A \models \varphi \text{ if and only if } \Gamma_e(\A) \models \varphi'.\]
  Moreover, for a nonzero $\alpha < \omega^{CK}_1$, if $\varphi'$ is $\Sigma^{pc}_{\alpha}$ (or $\Pi^{pc}_{\alpha}$), then so is $\varphi$.
\end{theorem}

Since Theorem~\ref{th:scott:pullback} can be relativized, it is straightforward to obtain a non-effective version.

\begin{corollary}
  Let $\Gamma \colon \mathfrak{K} \embedScott \mathfrak{K}'$.
  Then for any positive infinitary sentence $\varphi'$ in the signature of $\mathfrak{K}'$,
  there exists a positive infinitary sentence $\varphi$ in the signature of $\mathfrak{K}$ such that for all $\A \in \mathfrak{K}$,
  \[\A \models \varphi \text{ if and only if } \Gamma(\A) \models \varphi'.\]
  Moreover, for a nonzero $\alpha < \omega_1$, if $\varphi'$ is $\Sigma^{p}_{\alpha}$ (or $\Pi^{p}_{\alpha}$), then so is $\varphi$.
\end{corollary}

\section{Characterization of $\TxtEx$-Learnability}

Recall the class $\tilde\E$ from Remark~\ref{remark:the-key-classes}. In this section, we obtain the following characterization of $\TxtEx$-learning:

\begin{theorem}\label{th:txtex-learning-equivalence}
  For a class $\mathfrak{K} = \{\B_i : i \in \omega\}$, the following are equivalent:
  \begin{enumerate}[(1)]
  \item
    The class $\mathfrak{K}$ is $\TxtEx$-learnable.
  \item
    $\mathfrak{K} \embedScott \tilde\E$.
  \item
    There is a sequence of $\Sigma^{p}_2$ sentences $\{\psi_i : i \in \omega\}$ such that for all $i$ and $j$,
    $\B_j \models \psi_i$ if and only if $i = j$.
  \end{enumerate}
\end{theorem}

The proof of Theorem~\ref{th:txtex-learning-equivalence} is given as a sequence of lemmas.

\textbf{(1)}$\Rightarrow$\textbf{(2)}. For a finite sequence $\sigma$ and an $L$-structure $\A$, we say that $\sigma$ \emph{is on} $\D_+(\A)$ if $\sigma$ is an initial segment of some text for the structure $\A$.

\begin{proposition}\label{pr:txt-ex-tell-tale}
  Let $\mathfrak{K}$ be a class of structures which is $\TxtEx$-learnable by a learner $M$, and let $\A$ be a structure in $\mathfrak{K}$.
  For any finite $\sigma$ on $\Dp(\A)$, there exists an extension $\sigma'$ of $\sigma$ on $\Dp(\A)$,
  such that for all $\tau$ on $\Dp(\A)$, extending $\sigma'$, $M(\sigma') = M(\tau)$.
\end{proposition}
\begin{proof}
  Assume that there exists $\sigma$ on $\Dp(\A)$ such that for all $\sigma' \succeq \sigma$, we can find $\tau \succ \sigma'$
  such that $M(\sigma') \neq M(\tau)$.
  In this way it is clear that we can build an enumeration $\mathbb{T}$ of $\Dp(\A)$ such that $\lim_{n\to\infty} M(\mathbb{T}\upharpoonright{n})$ does not exist, and we obtain a contradiction.
  \qed 
\end{proof}

\begin{lemma}\label{lem:txtex-embeds-tilde-eq}
  For any $\TxtEx$-learnable class $\mathfrak{K}$, $\mathfrak{K}~\embedScott~\tilde\E$.
\end{lemma}
\begin{proof}
  For any finite enumeration $\sigma\in\omega^{<\omega}$ and ordinal $\alpha \leq \omega$, we define the auxiliary equivalence structure
  $\A_{\sigma,\alpha}$ with domain consisting of the elements $\{x_{\sigma,k} : k < \alpha\} \cup \{y_{\sigma,k} : k < \omega\}$
  and $\Dp(\A_{\sigma,\alpha})$ saying that the elements $x_{\sigma,k}$ form an equivalence class of size at least $\alpha$ and the elements $y_{\sigma,k}$ form an equivalence class of size $\omega$.
  Notice that $k \leq m$ implies $\Dp(\A_{\sigma,k}) \subseteq \Dp(\A_{\sigma,m})$.

  Suppose $M$ is a learner for the class $\mathfrak{K}$. We describe how the
  desired mapping $\Gamma$ works.
  Given a finite set $D$, consider all finite enumerations $\sigma$ of parts of $D$.
  First, we make sure that if $M(\sigma) = i$, then $\Gamma(D)$ contains the positive diagram of $\A_{\sigma,i+1}$.
  Second, if for some initial segment $\tau$ of $\sigma$, $M(\tau) \neq M(\sigma)$, then $\Gamma(D)$ contains the positive diagram of $\A_{\tau,\omega}$.
  More formally, let $E_D$ be the set of all $\sigma$ enumerating parts of $D$. Then $\Gamma(D)$ is the least set obeying the rules:
  \begin{enumerate}[(1)]
  \item
    $\bigcup_{\sigma\in E_D}\{\Dp(\A_{\sigma,i+1}) : M(\sigma) = i\} \subseteq \Gamma(D)$;
  \item
    $\bigcup_{\sigma\in E_D}\{\Dp(\A_{\tau,\omega}) : \tau \prec \sigma\ \&\ M(\tau) \neq M(\sigma)\} \subseteq \Gamma(D)$.
  \end{enumerate}
  It is easy to see that $\Gamma$ is motonone and compact, and hence Scott-continuous by Corollary~\ref{cor:scott-continuous:characterization}.
  
  Let $\B_i \in \mathfrak{K}$. We will show that $\Gamma(\B_i)$ is an equivalence structure of type $[\omega:\omega, i+1:\omega]$.
  By Proposition~\ref{pr:txt-ex-tell-tale}, there are infinitely many $\sigma$ on $\Dp(\B_i)$ such that $M(\sigma) = i$ and for all $\tau \succ \sigma$, $M(\tau) = i$.
  By the construction of $\Gamma$, $\Gamma(\B_i)$ contains infinitely many equivalence classes of size $i+1$.
  Assume that $\Gamma(\B_i)$ contains an equivalence class of a finite size $j+1 \neq i+1$. This can happen if there is a finite $\rho$ on $D_+(\B_i)$
  such that $M(\rho) = j$ and the equivalence structure $\A_{\rho,j+1}$ is a part of $\Gamma(\B_i)$. Again by Proposition~\ref{pr:txt-ex-tell-tale}, there exists an extension $\rho'$ of $\rho$
  such that $M(\rho') = M(\tau)$ for all $\tau$ extending $\rho'$. Since $M$ learns $\B_i$, we have $M(\rho') = i$. It follows by the construction of $\Gamma$ that the equivalence structure $\A_{\rho,j+1}$ is extended to $\A_{\rho,\omega}$ in $\Gamma(\B)$.
  \qed 
\end{proof}

\textbf{(2)}$\Rightarrow$\textbf{(3)}. The next proposition shows the usefulness of Theorem~\ref{th:scott:pullback} in giving a syntactic characterization of $\TxtEx$-learnable classes.

\begin{lemma}
  Let $\mathfrak{K} = \{\B_i : i < \omega\}$ be a class such that $\Gamma\colon \mathfrak{K} \embedScott \tilde\E$. 
  Then there exist $\Sigma^{p}_2$ sentences $\varphi_i$ such that
  $\B_i \models \varphi_j$ if and only if $i = j$.
\end{lemma}
\begin{proof}
  Without loss of generality, we may assume that $\Gamma(\B_i)$ is an equivalence structure $\A_i$ of type $[\omega : \omega, i+1 : \omega]$. For $\mathcal{A}_i$, 
  we have the infinitary $\Sigma^p_2$ sentence
  \begin{equation}\label{equ:varphi_i}
  	\varphi_i \df \exists x_0\cdots x_i\bigg[\bigwedge_{k\neq\ell\leq i} (x_{k}
      \sim x_{\ell}\mathrel{\&} x_{k} \neq x_{\ell})\mathrel{\&} \forall y(\neg y \sim x_0 \lor \bigvee_{\ell \leq i} \neg y \neq x_\ell)\bigg].
	\end{equation}
  Notice that we assume that $\neq$ is in our signature, so $x_k \neq x_{\ell}$ is a positive atomic formula.
  By Theorem~\ref{th:scott:pullback}, we obtain $\Sigma^{p}_2$ sentences for the structures $\B_i$ in $\mathfrak{K}$. 
  \qed
\end{proof}

\textbf{(3)}$\Rightarrow$\textbf{(1)}. We give the final part of the proof:

\begin{lemma}\label{lemma:3-to-1}
  Let $\mathfrak{K} = \{\B_i : i < \omega\}$.
  Suppose that there exist $\Sigma^{p}_2$ sentences $\varphi_i$ such that 
  $\B_i \models \varphi_j$ if and only if $i = j$.
  Then $\mathfrak{K}$ is $\TxtEx$-learnable.
\end{lemma}
\begin{proof}
  Without loss of generality, suppose that the $\Sigma^p_2$ sentence $\varphi_i$ has the form
  \[\varphi_i = \exists \bar{x}_i \bigg( \alpha_i(\bar{x}_i) \land \bigwedge_{j\in J_i}\forall\bar{y}_j \neg(\beta_{i,j}(\bar{x}_i,\bar{y}_j)) \bigg),\]
  where $\alpha_i$ and $\beta_{i,j}$ are positive atomic formulas, and $J_i$ is a countable set.
  We will describe how the learner $M$ for the class $\mathfrak{K}$ works.
  
  Consider an arbitrary sequence $\sigma$ of positive atomic formulas. We must determine the value of $M(\sigma)$.
  We find the least $\pair{i,\bar{a}}$ such that the G\"odel code of 
  $\alpha_i(\bar{a})$ is in the range of $\sigma$ and no sentence of the form $\beta_{i,j}(\bar{a},\bar{b}_j)$ is in the range of $\sigma$.
  Then we let $M(\sigma) = i$.

  Suppose $\A \cong \B_i$ and consider some text $\mathbb{T}_\A$ for $\mathcal{A}$.
  Since $\A \models \varphi_i$, find the least tuple $\bar{a}$ such that $\A \models \alpha_i(\bar{a})$ and
  $\A \models \bigwedge_{j\in J_i}\forall\bar{y}_j \neg(\beta_{i,j}(\bar{a},\bar{y}_j))$.
  It follows that the code of $\alpha_i(\bar{a})$ will appear in some initial segment of $\mathbb{T}_\A$
  and none of the positive atomic sentences $\beta_{i,j}(\bar{a},\bar{b}_j)$ will appear in $\mathbb{T}_\A$ for any $\bar{b}_j$ and any $j \in J_i$.
  It follows that $\lim_{n\to\infty} M(\mathbb{T}_\A \upharpoonright {n}) = i$.
  
  Lemma~\ref{lemma:3-to-1} and Theorem~\ref{th:txtex-learning-equivalence} are proved.
  \qed
\end{proof}


The choice of the class $\tilde\E$ in Theorem~\ref{th:txtex-learning-equivalence} seems somewhat arbitrary. The statement of Theorem~\ref{th:txtex-learning-equivalence} suggests the following definition.

\begin{definition}
  A countably infinite class $\mathfrak{K}_0$ is \emph{$\TxtEx$-complete} if 
  \begin{itemize}
  \item
    $\mathfrak{K}_0$ is $\TxtEx$-learnable, and
  \item
    for any countable $\TxtEx$-learnable class $\mathfrak{K}$, $\mathfrak{K} \embedScott \mathfrak{K}_0$.
  \end{itemize}
\end{definition}

\begin{corollary}\label{cor:tEcomplete}
  The class $\tilde \E$ is $\TxtEx$-complete.
\end{corollary}

\section{$\TxtEx$-Complete Classes}

Consider a signature $L_{\text{st}} = \{<,=,\neq\} \cup \{P_i : i \in \omega\}$,
where all $P_i$ are unary.
For each $i \in \omega$, we define a structure $\A_i$, where all $P^{\A_i}_j$ are disjoint infinite sets.
In addition, for any two elements $x,y$, where $x \in P_j$ and $y \in P_k$ for $j\neq k$, $x$ and $y$ are incomparable under $<$.
Let $\eta$ denote the order type of the rationals, and, if $\A_{i,j}$ is the
restriction of $\A_i$ to the elements in $P_j$, then define
\[\A_{i,j} \cong
  \begin{cases}
    \eta, & \text{if } i \neq j\\
    1+\eta, & \text{if }i = j.
  \end{cases}
\]

Let us denote $\mathfrak{K}_{\text{st}} = \{\A_i : i \in \omega\}$.
This class is studied in \cite{bazhenov_learning_2020}, and by combining \cite{bazhenov_learning_2020} with Corollary~\ref{cor:cantor:pullback}, the following characterization of $\InfEx$-learnability is obtained.

\begin{theorem}[\cite{bazhenov_learning_2020}]
  For a class $\mathfrak{K} = \{\B_i : i \in \omega\}$, the following are equivalent:
  \begin{enumerate}[(1)]
  \item
    The class $\mathfrak{K}$ is $\InfEx$-learnable.
  \item
    $\mathfrak{K} \embedCantor \mathfrak{K}_{\text{st}}$.
  \item
    There is a sequence of $\Sigma^{\text{inf}}_2$ sentences $\{\psi_i : i \in \omega\}$ such that for all $i$ and $j$,
    $\B_j \models \psi_i$ if and only if $i = j$.
  \end{enumerate}
\end{theorem}

This result suggests the following definition.

\begin{definition}
  We say that a class $\mathfrak{K}_0$ is $\InfEx$-complete if 
  \begin{itemize}
  \item
    $\mathfrak{K}_0$ is $\InfEx$-learnable;
  \item
    for any $\InfEx$-learnable class $\mathfrak{K}$, $\mathfrak{K} \embedCantor \mathfrak{K}_0$.
  \end{itemize}
\end{definition}

It follows that the class $\mathfrak{K}_{\text{st}}$ is $\InfEx$-complete. Now
we will show that $\mathfrak{K}_{\text{st}}$ is also $\TxtEx$-complete (see
Proposition~\ref{prop:Complete} below).
To do this, we borrow some ideas from \cite{bazhenov_computable_2021} to get a series of ancillary facts.

\begin{lemma}
  $\tilde\E \leq_{pc} \mathfrak{K}_{\text{st}}$.
\end{lemma}
\begin{proof}
  For each number $k\geq 1$, consider the enumeration operator $\Gamma_{e_k}$ (which takes an equivalence structure as an input), where the domain of the output structure is the set of non-empty tuples
  \[D_k = \bigg\{ (x_0,\dots,x_n) : \bigwedge_{i<n}x_i <_{\mathbb{N}} x_{i+1}\ \&\ |[x_i]_\sim| \geq k+1\ \&\ |[x_n]_\sim| \geq k \bigg\},\]
  where the ordering between the tuples is given by $\ov{x} \prec \ov{y}$ if and
  only if $\ov{x}$ is a proper extension of $\ov{y}$, or
    $x_i <_{\mathbb{N}} y_i$ for some index $i < \min\{|{\ov{x}}|,|\ov{y}|\}$.

  If the input structure $\A$ has type $[\omega:\omega,\ k:\omega]$, then
  $\Gamma_{e_k}(\A)$ is a linear ordering with a least element and no greatest
  element, and, if the input structure $\A$ has type $[\omega:\omega,\ m:\omega]$ for $m \neq k$, then
  $\Gamma_{e_k}(\A)$ is a linear ordering with no least element and no greatest element.

  Now, let $\Psi_{a_k}$ be such that $\Psi_{a_k}(\A)$ enumerates a copy of $1+\eta$ in
  place of the elements enumerated by $\Gamma_{e_k}(\A)$. 
  At last, let $\Theta(\A)$ enumerate the disjoint union of the structures
  $\Gamma_{e_k}(\A)$ with $P_k$ distinguishing the substructure enumerated by
  $\Gamma_{e_k}(\A)$. It is now routine to check that $\tilde \E\leq_{pc}
  \mathfrak{K}_{\text{st}}$ via $\Theta$.
  \qed
\end{proof}

\begin{lemma}
  $\mathfrak{K}_{\text{st}} \leq_{pc} \E$.
\end{lemma}
\begin{proof}
  Suppose that the input structure $\A$ has domain $\{x_{k,i} : i,k\in \omega\}$, where $P^\A_k = \{x_{k,i} : i \in \omega\}$.
  We describe how the enumeration operator $\Gamma_e$ works.
  The output structures of $\Gamma_e$ will always have domain a subset of $\{y^j_{k,i} : k,i,j \in \omega\}$.
  
  For any $k$, on input the finite diagram $D_k$ describing a finite chain (inside $P^\A_k$)
  $x_{k,i_0} <_{\A} x_{k,i_1} <_{\A} \cdots <_{\A} x_{k,i_n}$,
  $\Gamma_e(D_k)$ is an infinite part of the output equivalence structure describing the following:
  \[\bigwedge_{j<k} (y^j_{k,i_0} \sim y^{j+1}_{k,i_0})\ \&\ \bigwedge^n_{\ell=1}\bigwedge_{j\in\omega} (y^{j}_{k,i_\ell} \sim y^{j+1}_{k,i_\ell}).\]
  In other words, we associate with the current least element in the $k$-th linear ordering $P^\A_k$ an equivalence class of size $k+1$, 
  and with any other element in the $k$-th linear ordering we associate an infinite equivalence class.

  Let $\A_k$ be the restriction of $\A$ to $P^\A_k$.
  If $\A_k \cong 1+\eta$, then $\Gamma_e(\A)$ will contain an equivalence class of size $k+1$, and all other equivalence classes will be infinite.
  \qed
\end{proof}

\begin{lemma}
  $\E \leq_{pc} \tilde\E$.
\end{lemma}
\begin{proof}
  We define $\Gamma \colon \E \leq_{pc} \tilde\E$ is a straightforward manner: 
  $\Gamma$ essentially copies the input structure infinitely many times.
  \qed
\end{proof}

By combining the previous three lemmas and Corollary~\ref{cor:tEcomplete} we
obtain: 

\begin{proposition}\label{prop:Complete}
  The classes $\E$, $\tilde\E$, and $\mathfrak{K}_{\text{st}}$ are $\TxtEx$-complete.
\end{proposition}

%

\section{Applications}

Recall that \cite[Theorem 1.4]{pmlr-v98-fokina19a} proves that the class $\mathfrak{K} = \{[\omega:1], [\omega:2]\}$ is $\InfEx$-learnable, but not $\TxtEx$-learnable.
We give a new simple proof of this fact using $\TxtEx$-complete classes.

\begin{proposition}
  The class $\mathfrak{K} = \{[\omega:1], [\omega:2]\}$ is $\InfEx$-learnable, but not $\TxtEx$-learnable.
\end{proposition}
\begin{proof}
  Towards a contradiction, assume that $\mathfrak{K}$ is $\TxtEx$-learnable.
  A simple analysis of the proof of Lemma~\ref{lem:txtex-embeds-tilde-eq} shows that for the class $\mathfrak{K}_0 = \{[\omega:1,1:1], [\omega:1,2:1]\}$,
  we must have $\mathfrak{K} \embedScott \mathfrak{K}_0$ via some Scott-continuous operator $\Gamma$.
  Without loss of generality, suppose that for any structure $\A$ of type $[\omega:1]$,
  $\Gamma(\A)$ is an equivalence structure of type $[\omega:1,1:1]$.
  
  Let $b_0$ be the element in $\Gamma(\A)$ such that $|[b_0]_\sim| = 1$.
  By compactness, there is some finite part $\alpha$ of $\A$ for which $b_0 \in \Gamma(\alpha)$.
  Now, partition $\A$ into two infinite classes of infinite size such that $\alpha$ is contained entirely in one of the classes. In this way we produce a structure $\A'$ of type $[\omega:2]$ which is a substructure (w.r.t.\ positive atomic facts) of $\A$. Since Scott-continuity implies monotonicity, $\Gamma(\A') \subseteq \Gamma(\A)$.
  Since $b_0 \in \Gamma(\alpha)$ and $\alpha$ is a finite part of $\A'$,
  $b_0 \in \Gamma(\A')$. But since $\Gamma(\A')$ has type $[\omega:1,2:1]$, it follows that
  there is at least one element $c_0$ such that $\Gamma(\A') \models b_0 \sim c_0$.
  Since $\Gamma(\A') \subseteq \Gamma(\A)$, it follows that $\Gamma(\A) \models b_0 \sim c_0$.
  We reach a contradiction with the fact that $|[b_0]_\sim| = 1$ in $\Gamma(\A)$. \qed
\end{proof}

It is natural to search for $\TxtEx$-learnable classes which are not $\TxtEx$-complete.
Consider the class $\mathfrak{K} = \{\A_i : i\geq 1\}$, where $\A_i$ is an equivalence structure of type $[i:\omega]$.
It is clear that $\mathfrak{K}$ is $\TxtEx$-learnable, since for each structure $\A_i$ we have a distinguishing $\Sigma^p_2$ sentence $\psi_i := \varphi_{i-1}$ taken from Eq.~(\ref{equ:varphi_i}).


To see that $\mathfrak{K}$ is not $\TxtEx$-complete, it is enough to consider the following.

\begin{proposition}\label{prop:not-txtex-complete}
  $\{[\omega:1,1:1],[\omega:1,2:1]\} \not\embedScott \{[1:\omega],[2:\omega]\}$.
\end{proposition}
\begin{proof}
  Assume that $\Gamma\colon \{[\omega:1,1:1], [\omega:1,2:1]\} \embedScott \{[1:\omega],[2:\omega]\}$. 
  Let $\B$ be an equivalence structure of type $[\omega:1,1:1]$ and $\A$ be a substructure (w.r.t.\ positive atomic facts) of $\B$ of type $[\omega:1,2:1]$.
  By the monotonicity of Scott-continuous operators, $\Gamma(\A) \subseteq \Gamma(\B)$.
  But $\Gamma(\A)$ is an equivalence structure of type $[2:\omega]$ and $\Gamma(\B)$ is an equivalence structure of type $[1:\omega]$.
  We reach a contradiction by observing that $[2:\omega]$ is not embeddable into $[1:\omega]$.
  \qed 
\end{proof}

Assume that the class $\mathfrak{K}$ is $\TxtEx$-complete, then $\E \embedScott \mathfrak{K}$.
We can easily generalize the argument from Proposition~\ref{prop:not-txtex-complete} to reach a contradiction.

\section{Further Discussion}





Recall that the paper~\cite{BCSM-21} explored some connections between $\InfEx$-learnability and descriptive set theory. Here we elaborate more on this approach.

For $\alpha, \beta \in 2^\omega$, we define $\alpha \mathrel{E_0} \beta$ if and only if $(\exists n)(\forall m\geq n)[\alpha(n) = \beta(n)]$.
In \cite{BCSM-21}, a class of structures $\mathfrak{K}$ is characterized as
$\InfEx$-learnable if and only if the isomorphism relation $\cong
\,\upharpoonright\mathrm{LD}(\mathfrak{K})$ is (Cantor-)continuously reducible
to the relation $E_0$ of eventual agreement on reals (i.e., there is a
Cantor-continuous function $\Phi$ such that for all $\A,\B\in \mathfrak K$
$\A\cong \B$ if and only if $\Phi(\A) \mathrel{E_0} \Phi(\B)$ .


Motivated by this result we formulate the following question.
\begin{question}\label{question:1}
  Characterize $\TxtEx$-learnability in terms of Scott-continuous functions and familiar Borel equivalence relations.
\end{question}

\paragraph{{\bf Acknowledgements.}} {Fokina was supported by the Austrian Science Fund FWF through the project P~36781. Rossegger was supported by the European Union's Horizon 2020 Research and Innovation Programme under the Marie Sk\l{}odowska-Curie grant agreement No. 101026834 — ACOSE. Soskova and Vatev were partially supported by FNI-SU 80-10-180/17.05.2023.}



\printbibliography

\end{document}